\documentclass[reqno, a4paper, 10pt]{amsart}
\usepackage{amssymb,url, color, mathrsfs}
\usepackage[colorlinks=true, bookmarks=true, pdfstartview=FitH, pagebackref=true, linktocpage=true, linkcolor = magenta, citecolor = blue]{hyperref}
\usepackage[short,nodayofweek]{datetime}
\usepackage[pagewise,running,mathlines,displaymath, switch]{lineno}
\usepackage{times,enumerate}
\usepackage{float}

\theoremstyle{plain}
\numberwithin{equation}{section}
\newtheorem{theorem}{Theorem}[section]
\newtheorem{lemma}[theorem]{Lemma}

\theoremstyle{remark}

\DeclareMathOperator{\Rset}{\mathbf{R}}

\DeclareMathOperator{\fint}{\int\mkern-17.8mu-\mkern-0.0mu-}

\def\XXint#1#2#3{{\setbox0=\hbox{$#1{#2#3}{\int}$ }
\vcenter{\hbox{$#2#3$ }}\kern-.6\wd0}}

\allowdisplaybreaks

\allowdisplaybreaks

\definecolor{brown}{rgb}{0.5,0,0}
\definecolor{backgroundcolor}{rgb}{0.98, 0.92, 0.73}

\def\cfac#1{\ifmmode\setbox7\hbox{$\accent"5E#1$}\else\setbox7\hbox{\accent"5E#1}\penalty 10000\relax\fi\raise 1\ht7\hbox{\lower1.05ex\hbox to 1\wd7{\hss\accent"13\hss}}\penalty 10000\hskip-1\wd7\penalty 10000\box7 }

\author[Q.A. Ng\^o]{Qu\cfac oc Anh Ng\^o}
\address[Q.A. Ng\^o]{Department of Mathematics\\
College of Science, Vi\^{e}t Nam National University\\
H\`{a} N\^{o}i, Vi\^{e}t Nam.}
\email{\href{mailto: Q.A. Ng\^o <nqanh@vnu.edu.vn>}{nqanh@vnu.edu.vn}}
\email{\href{mailto: Q.A. Ng\^o <bookworm\_vn@yahoo.com>}{bookworm\_vn@yahoo.com}}

\begin{document}  

\setpagewiselinenumbers
\setlength\linenumbersep{100pt}

\title[Classification of $(-\Delta)^N u + u^{-(4N-1)}= 0$ with linear growth at infinity in $\Rset^{2N-1}$]
{Classification of entire solutions of $(-\Delta)^N u + u^{-(4N-1)}= 0$ with exact linear growth at infinity in $\Rset^{2N-1}$}

\begin{abstract}
In this paper, we study global positive $C^{2N}$-solutions of the geometrically interesting equation $(-\Delta)^N u + u^{-(4N-1)}= 0$ in $\Rset^{2N-1}$. We prove that any $C^{2N}$-solution $u$ of the equation having linear growth at infinity must satisfy the integral equation
\[
u(x) = c_0 \int_{\Rset^{2N-1}} {|x - y|{u^{-(4N-1)}}(y)dy} 
\]
for some positive constant $c_0$ and hence takes the following form
\[
u(x) = (1+|x|^2)^{1/2}
\]
in $\Rset^{2N-1}$ up to dilations and translations. We also provide several non-existence results for positive $C^{2N}$-solutions of $(-\Delta)^N u = u^{-(4N-1)}$ in $\Rset^{2N-1}$.
\end{abstract}

\date{\bf \today \; at \, \currenttime}



\maketitle

\section{Introduction}

In this paper, we are interested in classification of entire solutions of the following geometric interesting equation 
\begin{equation}\label{eqMAIN}
(-\Delta)^N u + u^{-(4N-1)} = 0
\end{equation}
in $\Rset^{2N-1}$ with $N \geqslant 2$. In order to understand the significance of studying Eq. \eqref{eqMAIN} and the reason why we work on this equation, let us briefly exploit its root in conformal geometry. Loosely speaking, equations of the form \eqref{eqMAIN} come from the problem of prescribing $Q$-curvature on $\mathbb S^{2N-1}$, which is associated with the conformally covariant GJMS operator with the principle part $ \Delta_g^{2N-1}$, discovered by Graham--Jenne--Mason--Sparling \cite{GJMS}. This operator is a high-order elliptic operator analogue with the well-known conformal Laplacian in the problem of prescribing scalar curvature. 

Given a dimensional constant $n \geqslant 3$, let us consider the model $(\mathbb S^n, g_{\mathbb S^n})$ equipped with the standard metric $g_{\mathbb S^n}$. In this case, it is well-known that the GJMS operator of order $2N$ with $N \geqslant 2$ is given by
\begin{equation}\label{eqGJMS}
P_{2N, g_{\mathbb S^n}} (\cdot)= \prod\limits_{k =1}^N { \Big( \Delta_{g_{\mathbb S^n}} - \Big( \frac n2 - k \Big) \Big( \frac n2 + k - 1 \Big) \Big)} .
\end{equation}
The GJMS operator \eqref{eqGJMS} is conformally covariant is the sense that if we conformally change the standard metric $g_{\mathbb S^n}$ to a new metric $\widetilde g$ via $\widetilde g = v^{4/(n-2N)}g_{\mathbb S^n}$ for some smooth function $v$ on $\mathbb S^n$, then the two operators $P_{2N, \widetilde g}$ and $P_{2N, g_{\mathbb S^n}}$ are related via
\begin{equation}\label{eqConformallyCovariantOfGJMS}
P_{2N, \widetilde g} (\varphi) = v ^{-\frac{n+2N}{n-2N}} P_{2N, g_{\mathbb S^n}} (v \varphi)
\end{equation}
for any smooth, positive function $\varphi$ on $\mathbb S^n$. In \eqref{eqConformallyCovariantOfGJMS} if we set $\varphi \equiv 1$, then we obtain
\[
P_{2N, g_{\mathbb S^n}} (v) = P_{2N, \widetilde g} (1) v^\frac{n+2N}{n-2N}. 
\]
Thanks to \cite[Eq. (1.12)]{juhl}, we know that 
\[
P_{2N, \widetilde g} (1) = (-1)^N \Big( \frac n2 - N \Big) Q_{2N, \widetilde g}.
\]
for some scalar function $Q_{2N, \widetilde g}$ known that the $Q$-curvature associated with the GJMS operator $P_{2N, \widetilde g} $. From this we obtain the equation
\begin{equation}\label{eqQCurvature}
P_{2N, g_{\mathbb S^n}} (v) = (-1)^N \Big( \frac n2 - N \Big) Q_{2N, \widetilde g} \, v^\frac{n+2N}{n-2N} .
\end{equation}
Let us now limit ourselves to the case $n=2N-1$. Then up to a multiple of positive constants, Eq. \eqref{eqQCurvature} becomes
\begin{equation}\label{eqQCurvatureSpecial}
P_{2N, g_{\mathbb S^n}} (v) = (-1)^{N-1} Q_{2N, \widetilde g} \, v^\frac{n+2N}{n-2N} .
\end{equation}
Toward understanding the structure of the solution set of Eq. \eqref{eqQCurvatureSpecial}, let us only consider the case when $Q_{2N, \widetilde g}$ is constant. Upon a suitable scaling, we may assume $Q_{2N, \widetilde g} = \pm 1$.
Therefore, Eq. \eqref{eqQCurvatureSpecial} becomes
\begin{equation}\label{eqQCurvature(2N-1)}
P_{2N, g_{\mathbb S^n}} (v) = \pm (-1)^{N-1} v^{-(4N-1)} .
\end{equation}
Let us now denote by $\pi : \mathbb S^{2N-1} \to \Rset^{2N-1}$ the stereographic projection and set
\begin{equation}\label{eqUVViaProjection}
u(x)=v(\pi^{-1}(x))\Big( \frac{1+|x|^2}{2}\Big)^{1/2}
\end{equation}
for $x \in \Rset^{2N-1}$. Thanks to \cite[Proposition 1]{Gra07}, we can project \eqref{eqGJMS} with $n=2N-1$ from $\mathbb S^{2N-1}$ to $\Rset^{2N-1}$ to get
\begin{equation}\label{eqGJMSAfterProjection}
\Big( \frac{2}{1+|x|^2}\Big)^{-\frac{4N-1}{2}} (\Delta^N u)(x) = P_{2N, g_{\mathbb S^n}} \big( v(\pi^{-1}(x)) \big) .
\end{equation}
Therefore, via the stereographic projection $\pi$ and up to a multiplication of positive constant, combining Eq. \eqref{eqGJMSAfterProjection} and Eq. \eqref{eqQCurvature(2N-1)} gives
\[
\Delta^N u = \pm (-1)^{N-1} u^{-(4N-1)}.
\]
In the preceeding equation, if we consider the plus sign, the resulting equation leads us to Eq. \eqref{eqMAIN} while for the minus sign, we arrive at the equation
\begin{equation}\label{eqMAIN2}
(-\Delta)^N u = u^{-(4N-1)}
\end{equation}
in $\Rset^{2N-1}$.

As far as we know, several special cases of Eq. \eqref{eqMAIN} have already been studied in the literature. To be precise, when $N=2$, the following equation
\begin{equation}\label{eqMAIN-3D}
\Delta^2 u + u^{-7} = 0
\end{equation}
in $\Rset^3$ was studied by Choi and Xu in \cite{ChoiXu} as well as by McKenna and Reichel in \cite{KR}. The main result in \cite{ChoiXu} is that if $u$ solves \eqref{eqMAIN-3D} with exact linear growth at infinity in the sense that $\lim_{|x| \to +\infty} u(x)/|x|$ exists then $u$ solves the following integral equation
\[
u(x) = \int_{\Rset^3} |x-y| u(y)^{-7}dy .
\]
From this integral representation, by a beautiful classification of positive solutions of integral equations by Li \cite{li2004} and Xu \cite{xu2005}, it is widely known that $u(x) = (1+|x|^2)^{1/2}$ up to dilations and translations. When $N=3$, Eq. \eqref{eqMAIN} leads us to the equation
\begin{equation}\label{eqMAIN-5D}
\Delta^3 u = u^{-11} 
\end{equation}
in $\Rset^5$. Its associated integral equation becomes
\[
u(x) = \int_{\Rset^5} |x-y| u(y)^{-11}dy .
\]
This integral equation was studied by Feng and Xu in \cite{fx2013}. The main result in \cite{fx2013} tell us that the only entire positive solution of Eq. \eqref{eqMAIN-5D} is $u(x) = (1+|x|^2)^{1/2}$ up to dilations and translations. As a counter-part of Eq. \eqref{eqMAIN-5D}, the following triharmonic Lane--Emden equation
\[
\Delta^3 u + |u|^{p-1}u = 0 
\]
in $\Rset^n$ with $p>1$ was recently studied by Luo, Wei, and Zou \cite{LWZ16}; see also \cite{GuoWei}. We take this chance to remind of a work by Ma and Wei in \cite{MaWei} where the authors studied the equation
\[
\Delta u = u^\tau
\]
with $\tau<0$. Clearly, this equation has a similar form of that of Eq. \eqref{eqMAIN} with $N=1$.

In the present paper, following the main question posted in \cite{ChoiXu, G}, we initiate our study on the structure of solution set of \eqref{eqMAIN} and \eqref{eqMAIN2}. To be precise, for Eq. \eqref{eqMAIN}, we are able to classify all solutions with exact linear growth at infinity. The following theorem is the content of this result.

\begin{theorem}\label{thmCLASSIFICATION}
All solutions of partial differential equation \eqref{eqMAIN} which satisfies
\begin{equation}\label{eqLinearGrowth}
\mathop {\lim }\limits_{|x| \to +\infty } \frac{{u(x)}}{{|x|}} = \alpha \quad \text{ uniformly }
\end{equation}
for some  \textbf{non-negative} finite constant $\alpha$ verify the following integral equation
\[
u(x) = c_0 \int_{\Rset^{2N-1}} {|x - y|{u^{-(4N-1)}}(y)dy} .
\]
Consequently, up to dilations and translations, the only entire solutions of \eqref{eqMAIN} satisfying \eqref{eqLinearGrowth} is
\[
u(x) = (1+|x|^2)^{1/2}
\]
in $\Rset^{2N-1}$.
\end{theorem}

As already discussed in \cite{ChoiXu}, a major reason for imposing assumption \eqref{eqLinearGrowth} in studying \eqref{eqMAIN} follows from the fact that entire solutions of \eqref{eqMAIN} with exact linear growth at infinity correspond to complete conformal metrics on $\mathbb S^{2N-1}$, thanks to \eqref{eqUVViaProjection}. We expect that Eq. \eqref{eqMAIN}, if freezing from geometric interpretation, also admits entire solutions with different growth at infinity. This is supported by considering Eq. \eqref{eqMAIN} when $N=2$; see \cite{G, DuocNgo2015}.

For Eq. \eqref{eqMAIN2}, we prove that in fact this equation does not admit solutions with exact linear growth at infinity.

\begin{theorem}\label{thmNONEXISTENCE}
There is no positive $C^{2N}$-solution to Eq. \eqref{eqMAIN2} which satisfies
\[
\lim\limits_{|x| \to +\infty } \frac{{u(x)}}{{|x|}} = \alpha \quad \text{ uniformly }
\]
for some \textbf{positive} finite constant $\alpha$.
\end{theorem}

We note that a similar non-existence result for solutions of Eq. \eqref{eqMAIN2} was obtained by Xu and Yang in \cite[Lemma 4.3]{XuYang2002}. To be exact, it was proved in \cite{XuYang2002} that there is no $C^4$-solution $u$ of \eqref{eqMAIN2} with $N=2$ in $\Rset^3$ which is bounded from below away from zero with the following conditions: $\int_{\Rset^3} u^{-6} dx <+\infty$, $\int_{\Rset^3} (\Delta u)^2 dx <+\infty$. In the following result, we generalize this result for solutions of \eqref{eqMAIN2}.

\begin{theorem}\label{thmNewNONEXISTENCE}
There is no positive $C^{2N}$-solution $u$ to Eq. \eqref{eqMAIN2} which satisfies
\begin{itemize}
\item [(1)] $\int_{\Rset^{2N-1}} u^{-(4N-2)} dx < \infty$, 
\item [(2)]$u \geqslant 1$ and $u(0)=1$, and
\item [(3)] $(-\Delta)^i u \in L^2(\Rset^{2N-1})$ for all $i=1,2,\dots, N-1$.
\end{itemize}
\end{theorem}

As in \cite{XuYang2002}, the main ingredient in the proof of Theorem \ref{thmNewNONEXISTENCE} are mean value properties for biharmonic functions and the Liouville theorem. Note that in the proof of Theorem \ref{thmNONEXISTENCE}, we exploit the super poly-harmonic property for solutions of Eq. \eqref{eqMAIN2} under the linear growth assumption. In the proof of Theorem \ref{thmNewNONEXISTENCE} we also exploit  the super poly-harmonic property for solutions of Eq. \eqref{eqMAIN2} without using the linear growth property.

In the next section, several fundamental estimates for solutions of Eq. \eqref{eqMAIN} are provided. These estimates are useful for obtaining integral representation for all $(-\Delta)^k u$ for $k$ from $N-1$ down to $0$. Once we have an integral representation for $u$, we are able to classify solutions. In the last part of the paper, we prove Theorems \ref{thmCLASSIFICATION}, \ref{thmNONEXISTENCE}, and \ref{thmNewNONEXISTENCE}.

\section{Elementary estimates}

In this section, we setup some notations and provide elementaty estimates necessary to deal with elliptic equations with poly-harmonic operators. We note that although our approach is similar to the one used in \cite{ChoiXu}, in several places, we have to introduce new ideas to deal with high-order elliptic equations. 

We will denote the sphere in $\Rset^{2N-1}$ of radius $r$ and center $x_0$ by $\partial B(x_0, r)$ and its included solid ball in $\Rset^{2N-1}$ by $B(x_0, r)$. We introduce the average of a function $f$ on $\partial B(x_0, r)$ by
\[
\overline f (x_0, r) = \frac{1}{\omega_{2N-1} r^{2N-2}}\int_{\partial B(x_0, r)} {f(x) d\sigma_x  } =\fint_{\partial B(x_0, r)} f(x) d\sigma_x
\]
which depends only on the radius $r$. Here by $\omega_{2N-1}$ we mean the volume of the unit sphere $\partial B(x_0, 1)$ centered at $x_0$ sitting in $\Rset^{2N-1}$. (Note that $\omega_n = 2\pi^{n/2}/\Gamma (n/2)$ for all $n$.) Throughout the paper, if $x_0 =O$, then we drop $O$ in the notation $\overline f(O, r)$ for simplicity.

We also denote various dimensional constants
\begin{equation}\label{eqConstantsC}
\left\{
\begin{split}
c_{N-1} & = \omega_{2N-1}^{-1}, \\
c_{N-k-1} &= \frac { c_{N-k} }{ 2k(2N-2k-3) } \quad \text{ for } 1 \leqslant k \leqslant N-2.
\end{split}
\right.
\end{equation}
Clearly $c_k>0$ for all $1 \leqslant k \leqslant N-2$. We also let $c_0>0$ be
\begin{equation}\label{eqConstantsC0}
c_0 = \frac{c_1}{2N-2}.
\end{equation}
Keep in mind that $-c_{N-1}|x-y|^{-(2N-3)}$ is the Green function of the operator $\Delta$ in $\Rset^{2N-1}$.

We list here the following useful inequality whose proof is exactly the same as \cite[Lemma 2.1]{ChoiXu} in $\Rset^3$.

\begin{lemma}\label{lem-ApplicationJensenInequality}
For any point $x_0$ in $\Rset^{2N-1}$ and any $q, r>0$, there holds
\[
\Big( {\fint_{ \partial B(x_0, r) } {fd\sigma } } \Big)^{ - q} \leqslant \fint_{ \partial B(x_0, r) } {{f^{ - q}}d\sigma } .
\]
\end{lemma} 

Using Lemma \ref{lem-ApplicationJensenInequality}, we obtain from \eqref{eqMAIN} the following differential inequality
\begin{equation}\label{eqMAINafterAverage}
(-\Delta)^N \overline u + \overline u^{-(4N-1)} \leqslant 0.
\end{equation}
In particular, there holds $(-\Delta)^N \overline u < 0$ everywhere in $\Rset^{2N-1}$. The next lemma, which is known as the sub poly-harmonic property of $u$, is of crucial importance as it allows to deal with high order equations.

\begin{lemma}\label{lem-PolySubHarmonic}
All positive solutions $u$ of \eqref{eqMAIN} with the growth \eqref{eqLinearGrowth} satisfy
\[(-\Delta)^k u < 0\]
everywhere in $\Rset^{2N-1}$ for each $k=1,...,N-1$.
\end{lemma}

\begin{proof}
This lemma can be proved by using a general result from \cite[Theorem 2]{ngo2016}. In the present scenario, its proof is rather simple and for completeness, we outline its proof. For $i \in \{1, ..., N-1\}$, suppose that 
\[
\sup_{\Rset^{2N-1}} (-\Delta)^{N-i} u \geqslant 0 > \sup_{\Rset^{2N-1}} \big \{ (-\Delta)^{N-i+1} u , ..., (-\Delta)^N u \big\}.
\]
If $(-\Delta)^{N-i} u \leqslant 0$ then there is a point $x_0 \in \Rset^{2N-1}$ such that $(-\Delta)^{N-i} u(x_0) = 0$. This implies that $x_0$ is a maximum point of $(-\Delta)^{N-i} u$; hence we must have 
\[
-\Delta (-\Delta)^{N-i} u (x_0) \geqslant 0
\]
which contradicts with our assumption $(-\Delta)^{N-i+1} u <0$. In the case $(-\Delta)^{N-i} u > 0$ everywhere in $\Rset^{2N-1}$, by induction and integration by parts, after taking the spherical average over $\partial B(0, r)$ we obtain
\begin{equation}\label{eqEstimateAverageV_p-i}
(-1)^i (-\Delta)^{N-i} \overline u  (r) \leqslant (-1)^i  (-\Delta)^{N-i} \overline u  (0) +  \sum_{l=1}^{i-1} \frac{   (-1)^{i-l}        (-\Delta)^{N-i+l} \overline u    (0) r^{2l}     }{\prod_{k=1}^{l} (2k) \prod_{k=1}^{l} [2N-1+2(k-1)] }
\end{equation}
for each $1 \leqslant i \leqslant N$ and for any $r$. In the special case $i=N$, we obtain from \eqref{eqEstimateAverageV_p-i} the following
\begin{equation}\label{eqEstimateForUbar}
(-1)^N \overline u (r) \leqslant (-1)^N  \overline u (0) +  \sum_{l=1}^{N-1} \frac{   (-1)^{N-l} (-\Delta)^l \overline u(0) r^{2l}     }{\prod_{k=1}^{l} (2k) \prod_{k=1}^{l} [2N-1+2(k-1)] }.
\end{equation}
From this we obtain a contradiction since $\overline u$ has linear growth at infinity and the leading coefficient on the right hand side of \eqref{eqEstimateForUbar} has a sign.
\end{proof}

In the rest of this section, we show how important Lemma \ref{lem-PolySubHarmonic} is by exploiting further properties of solutions of Eq. \eqref{eqMAIN}. First, we recall the following lemma in $\Rset^n$ instead of $\Rset^{2N-1}$.

\begin{lemma}\label{lem-BarrierForPolySubHarmonic}
Let $w$ be a radially symmetric, non-positive function satisfying
\[(-\Delta)^k w \leqslant 0\]
everywhere in $\Rset^n$ for each $k=1,...,m$ with $n < 2m$. Then necessarily we have
\[
r w'(r) + (n-2m) w(r) \leqslant 0, \quad r w''(r) + (n+1-2m)w'(r) \geqslant 0
\]
everywhere in $\Rset^n$. 
\end{lemma}

\begin{proof}
See \cite[Example 2.3]{CMM}.
\end{proof}

Using Lemma \ref{lem-BarrierForPolySubHarmonic} we can prove that $\overline u''$ has a sign. Such a result has some role in our analysis. In particular, this helps us to deduce that any solution of Eq. \eqref{eqMAIN} must grow at least linearly at infinity; see Lemma \ref{lem-AlphaIsPositive} below.

\begin{lemma}\label{lem-SecondDerivativeIsPositive}
All positive solutions $u$ of \eqref{eqMAIN} with the growth \eqref{eqLinearGrowth} satisfy 
$$\overline u'' (r) \geqslant 0$$
for any $r>0$.
\end{lemma}

\begin{proof}
Thanks to Lemma \ref{lem-PolySubHarmonic} and Eq. \eqref{eqMAIN}, with $v=-\Delta \overline u$ there holds
\[(-\Delta)^k v< 0\]
everywhere in $\Rset^{2N-1}$ for any $k=1,...,N-1$. Since $v$ is bounded from above by zero, we can apply Lemma \ref{lem-BarrierForPolySubHarmonic} to get
\[
r v'(r) + v(r) \leqslant  0, \quad r v''(r) + 2 v'(r) \geqslant 0.
\]
Using the formula $-r^{2-2N}(r^{2N-2} \overline u')' = v$ and the inequality $r v' + v <  0$, we deduce that
\[\begin{split}
-(r^{2N-2} \overline u')'' = &(r^{2N-2}v)' = (2N-2)r^{2N-3} v  + r^{2N-2} v'\\
=& r^{2N-3} \big( rv' + v\big) + (2N-3)r^{2N-3} v\\
\leqslant & (2N-3)r^{2N-3} v.
\end{split}\]
Thus, we have just proved that
\[
-(r^{2N-2} \overline u')'' \leqslant  (2N-3)r^{2N-3} v = (2N-3)r^{-1} (-r^{2N-2} \overline u')'.
\]
Therefore, if we set $w = r^{2N-2} \overline u'$, then we obtain
\[
(-rw' + (2N-2)w)' = -rw''+ (2N-3) w' \leqslant 0.
\]
By definition, the function $rw' - (2N-2)w$ vanishes at $r=0$ and is strictly increasing on $(0, +\infty)$. It follows that
\begin{equation}\label{eq-SecondDerivativeIsPositive-1}
r (r^{2N-2} \overline u')' \geqslant (2N-2)r^{2N-2} \overline u'
\end{equation}
for any $r \geqslant 0$, which is equivalent to
\[
r^{2N-1} \overline u'' +(2N-2)r^{2N-2} \overline u' \geqslant (2N-2)r^{2N-2} \overline u'
\]
for any $r \geqslant 0$. Hence, there holds
\begin{equation}\label{eq-SecondDerivativeIsPositive-2}
r^{2N-1} \overline u''  \geqslant 0.
\end{equation}
Hence $\overline u''(r) \geqslant 0$ for all $r>0$ as claimed.
\end{proof}

In the following lemma, we study the asymptotic behavior of $(-\Delta)^k \overline u$ at infinity. Such a result is useful when we apply the Liouville theory to get integral representation for $(-\Delta)^k u$.

\begin{lemma}\label{lem-Limits}
All positive solutions $u$ of \eqref{eqMAIN} with the growth \eqref{eqLinearGrowth} satisfy
\[\lim_{r \to +\infty} (-\Delta)^k \overline u(r) = 0\]
for each $k=1,...,N-1$.
\end{lemma}

\begin{proof}
Fix $k \in \{1,..., N-1\}$ and denote
\[
v_k (r) := (-\Delta)^k \overline u.
\]
For clarity, we also set $v_0 = \overline u$. Our aim is to prove that $v_k \to 0$ at infinity for each $k>0$. In view of Lemma \ref{lem-PolySubHarmonic}, there holds $v_k < 0$. Observe that in $\Rset^{2N-1}$ we have
\[
r^{2-2N} (r^{2N-2} v_k' )' = \Delta v_k = -(-\Delta)^{k+1} \overline u >0,
\]
which implies that $v_k' >0$. Therefore, $v_k$ has a limit at infinity. 

To prove the desired limit, let us start with $k=1$. Upon using our convention and the monotone decreasing of $-v_1$, we clearly have
\[
r^{2N-2} v_0' (r) = -\int_0^r s^{2N-2} v_1(s) ds \geqslant -\frac{r^{2N-1}}{2N-1} v_1 (r),
\]
which yields
\[
 v_0(r) \geqslant v_0(0) - C r^2 v_1 (r) \geqslant v_0(0)
\]
for some constant $C>0$. Since $v_0$ has linear growth at infinity, we deduce that $v_1 (r) \to 0$ as $r \to +\infty$. The above argument can be repeatedly used to conclude desired limits. Indeed, suppose that $v_{k-1} (r) \to 0$ as $r \to +\infty$, we will show that $v_k (r) \to 0$ a $r \to +\infty$. To this purpose, we observe that
\[
r^{2N-2} v_{k-1}' (r) = -\int_0^r s^{2N-2} v_k(s) ds \geqslant -\frac{r^{2N-1}}{2N-1} v_k (r),
\]
which implies
\[
 v_{k-1}(r) \geqslant v_{k-1}(0) - Cr^2 v_k (r) \geqslant v_{k-1}(0)
\]
for some constant $C >0$ which depends only on $N$. Dividing both sides by $r^2$, we obtain
\[
 \frac{v_{k-1}(r)}{r^2} \geqslant \frac{v_{k-1}(0)}{r^2} + C_1 (-v_k (r)) \geqslant \frac{v_{k-1}(0)}{r^2}.
\]
We now send $r \to +\infty$ to get the desired result.
\end{proof}

\begin{lemma}\label{lem-AlphaIsPositive}
Let $u>0$ satisfy \eqref{eqMAIN} with the linear growth \eqref{eqLinearGrowth}. Then $\alpha>0$ where the constant $\alpha$ is given in \eqref{eqLinearGrowth}.
\end{lemma}

\begin{proof}
In view of Lemma \ref{lem-SecondDerivativeIsPositive}, the inequality $\overline u'' (r) \geqslant 0$ implies that $\overline u' (r) \geqslant \overline u' (1) > 0$ for any $r \geqslant 1$. From this we obtain 
\[
 \overline u (r) \geqslant \overline u' (1) (r - 1) + \overline u (1)
\]
for all $r \geqslant 1$. The above inequality tells us that $u$ grows at least linearly at infinity, moreover, if the limit $\lim_{|x| \to +\infty} u(x)/|x| = \alpha \geqslant 0$ exists uniformly, it must hold $\alpha>0$ thanks to $\overline u'(1)>0$.
\end{proof}

\section{A classification result: Proof of Theorem \ref{thmCLASSIFICATION}}

The main purpose in this section is to provide a proof of Theorem \ref{thmCLASSIFICATION}. First if we set
 \[
 U(x) = c_0 \int_{\Rset^{2N-1}} {|x - y| u^{-(4N-1)}(y)dy} .
 \]
with the constant $c_0 > 0$ given by \eqref{eqConstantsC0}. Note that by the definition of the constants $c_i$ in \eqref{eqConstantsC}, there holds
\[
c_{N-k-1} \Delta_x (|x-y|^{2k-2N+3}) = -c_{N-k} |x-y|^{2k-2N+1} 
\] 
Therefore, an easy calculation shows that
\begin{equation}\label{eqDelta^kU}
(-\Delta)^k U(x) = - c_k \int_{\Rset^{2N-1}} \frac{ u^{-(4N-1)}(y) }{ |x-y|^{2k-1} } dy
\end{equation}
for $k = 1,.., N-1$ with the constant $c_k > 0$ given by \eqref{eqConstantsC} and
\[
(-\Delta)^N U(x) = - u^{-(4N-1)}.
\]
In particular, 
\[
(-\Delta)^k U(x)<0
\]
everywhere on $\Rset^{2N-1}$. Recall that the function $u$ solves $\Delta^N u = (-1)^{N-1} u^{-(4N-1)}$ in $\Rset^{2N-1}$. For simplicity, we set
\[
U_k (x) =(-\Delta)^k U(x) .
\]
We now prove the following important properties for $U_k$.

\begin{lemma}\label{lem-LimitDelta^iU}
For each fixed $k \in \{1, ..., N-1\}$, the function $U_k$ satisfies 
$$U_k (x) \to 0$$ 
as $|x| \to +\infty$.
\end{lemma}

\begin{proof}
It follows from \eqref{eqLinearGrowth} that there exists $R>0$ such that if $|x|>R$ then $u(x)> \alpha |x| /2$. This implies that
\[\int_{\Rset^{2N-1}} |x-y|^{2k} u^{ -(4N-1)}(y)dy < +\infty\]
for all $k=1,..., N-1$. In particular, we have $\int_{\Rset^{2N-1}} {u^{ - (4N-1)}(y)dy} < +\infty$ is finite and $u^{-(4N-1)}(x)$ is bounded function, say by $M>0$. By Eq. \eqref{eqDelta^kU}, we have
\[
U_k (x) = - c_k \int_{\Rset^{2N-1}} \frac{ u^{-(4N-1)}(y) }{ |x-y|^{2k-1} } dy.
\]
For given $\varepsilon>0$, there exists some $\delta>0$ small enough such that
\[
\int_{|x-y| \leqslant \delta} {\frac{{u^{ - (4N-1)}(y)}}{|x - y|^{2k-1}}dy} \leqslant CM \int_0^\delta s^{2N-2k-1} ds < \frac \varepsilon 2
\]
for any $x \in \Rset^{2N-1}$. In the region $\{ |x-y| \geqslant \delta \}$, we can use the dominated convergence theorem to conclude that 
\[
\lim_{|x| \to +\infty} \int_{|x-y| > \delta} {\frac{{u^{ - (4N-1)}(y)}}{|x - y|^{2k-1}}dy}=0.
\]
Therefore,
\[
\int_{|x-y|> \delta} {\frac{{u^{ - (4N-1)}(y)}}{|x - y|^{2k-1}}dy} < \frac \varepsilon 2
\]
for any large $x \in \Rset^{2N-1}$. This shows that $U_k (x)$ has the limit zero at infinity.
\end{proof}

Following the method used in \cite{ChoiXu}, to prove our main theorem, we need to establish an integral representation for $\Delta^k u$ for any $k \in \{1, ..., N-1\}$. First, for $\Delta^{N-1}u$, we prove the following result.

\begin{lemma}\label{lem-Delta^(N-1)}
Let $u$ satisfy \eqref{eqMAIN} with the linear growth \eqref{eqLinearGrowth}. Then the following representation
\begin{equation}\label{eqDelta^(N-1)U}
(-\Delta)^{N-1} u(x) = - c_{N-1} \int_{\Rset^{2N-1}} \frac{ u^{-(4N-1)}(y) }{ |x-y|^{2N-3} } dy
\end{equation}
holds with the constant $c_{N-1} > 0$ given in \eqref{eqConstantsC}.
\end{lemma}

\begin{proof}
Upon using the notation for $U_k$ mentioned at the beginning of this section, $U_{N-1}$ is exactly the right hand side of \eqref{eqDelta^(N-1)U}, that is
\[
U_{N-1} (x) = - c_{N-1} \int_{\Rset^{2N-1}} \frac{u^{-(4N-1)}(y)}{|x - y|^{2N-3}}dy.
\]
We also denote an upper bound of $u^{-(4N-1)}$ by $M$. By Lemma \ref{lem-LimitDelta^iU}, we know that $U_{N-1}$ is bounded. Note that $-c_{N-1}|x-y|^{-(2N-3)}$ is the Green function of $\Delta$ in $\Rset^{2N-1}$, therefore an easy calculation shows that
\[\begin{split}
\Delta U_{N-1} (x) =& \int_{\Rset^{2N-1}} \Delta_x \Big( \frac{-c_{N-1}}{|x - y|^{2N-3}} \Big) u^{-(4N-1)}(y) dy = u^{ - (4N-1)} (x).
\end{split}\]
Now it follows from the equations satisfied by $U_{N-1}$ and $u$ that
\[\Delta ((-\Delta)^{N-1}u - U_{N-1}) = 0\]
in $\Rset^{2N-1}$. Since $U_{N-1}$ is bounded and $(-\Delta)^{N-1}u$ is non-positive, we deduce that $(-\Delta)^{N-1} u-U_{N-1}$ is a harmonic function which is bounded either from above. Thus the Liouville theorem can be applied to conclude that
\begin{equation}\label{eqDelta^(N-1)U=} 
(-\Delta)^{N-1} u = U_{N-1} + b_{N-1}
\end{equation}
for some constant $b_{N-1}$. To get rid of the constant $ b_{N-1}$, we take the spherical average both sides of \eqref{eqDelta^(N-1)U=} to get
\[
v_{N-1} (r) = \overline U_{N-1} (r) + b_{N-1}
\]
where $v_{N-1}$ is defined in the proof of Lemma \ref{lem-Limits}. Taking the limit as $r \to +\infty$ we deduce that $b_{N-1}=0$, thanks to Lemmas \ref{lem-Limits} and \ref{lem-LimitDelta^iU}.
\end{proof}

By repeating the argument used in the proof of Lemma \ref{lem-Delta^(N-1)}, we easily obtain the following result for $\Delta^ku$ for each $k \in \{1,...,N-2\}$.

\begin{lemma}\label{lem-Delta^kU}
Let $u$ satisfy \eqref{eqMAIN} with the linear growth \eqref{eqLinearGrowth}. Then for each $k=1,...,N-1$, the following representation
\begin{equation}\label{eqDeltakU}
(-\Delta)^{N-k} u(x) = - c_{N-k} \int_{\Rset^{2N-1}} \frac{ u^{-(4N-1)}(y) }{ |x-y|^{2N-1-2k} } dy.
\end{equation}
holds with the constant $c_{N-k} >  0$ given in \eqref{eqConstantsC}.
\end{lemma}

\begin{proof}
We prove \eqref{eqDeltakU} by induction on $k$. Clearly \eqref{eqDeltakU} holds for $k=1$ by Lemma \ref{lem-Delta^(N-1)}. Suppose that \eqref{eqDeltakU} holds for $k$, that is
\[
(-\Delta)^{N-k} u(x) = -c_{N-k} \int_{\Rset^{2N-1}} \frac{ u^{-(4N-1)}(y) }{ |x-y|^{2N-1-2k} } dy 
\]
we prove \eqref{eqDeltakU} for $k+1$, that is
\[\begin{split}
(-\Delta)^{N-k-1} u(x) =&  - c_{N-k-1} \int_{\Rset^{2N-1}} \frac{ u^{-(4N-1)}(y) }{ |x-y|^{2N-3-2k} } dy. 
\end{split}\]
Notice that
\[\begin{split}
U_{N-k-1} (x) =&- c_{N-k-1} \int_{\Rset^{2N-1}} \frac{ u^{-(4N-1)}(y) }{ |x-y|^{2N-3-2k} } dy .
\end{split}\]
Clearly, the function $U_{N-k-1}$ is bounded by means of Lemma \ref{lem-LimitDelta^iU}. Hence
\[
\Delta U_{N-k-1} (x) = -c_{N-k-1} \int_{\Rset^{2N-1}} u^{-(4N-1)}(y) \Delta_x \Big( \frac{ 1 }{ |x-y|^{2N -3 - 2k} } \Big)dy.
\]
Note that by the definition of the constants $c_i$ in \eqref{eqConstantsC}, there holds
\[
c_{N-k-1} \Delta_x (|x-y|^{2k-2N+3}) = -c_{N-k} |x-y|^{2k-2N+1} 
\]
Therefore,
\[\Delta ((-\Delta)^{N-k-1}u - U_{N-k-1}) = 0\]
in $\Rset^{2N-1}$. Since $U_{N-k-1}$ is bounded and $(-\Delta)^{N-k-1}u$ is non-positive, we deduce that $(-\Delta)^{N-k-1} u -U_{N-k-1}$ is a harmonic function which is bounded either from above. Thus the Liouville theorem can be applied to conclude that
\begin{equation}\label{eqDelta^(N-k-1)U} 
(-\Delta)^{N-k-1} u = U_{N-k-1} + b_{N-k-1}
\end{equation}
for some constant $b_{N-k-1}$ Taking the spherical average both sides of \eqref{eqDelta^(N-k-1)U} to get
\[
v_{N-k-1} (r) = \overline U_{N-k-1} (r) + b_{N-k-1}
\]
where $v_{N-k-1}$ is defined in the proof of Lemma \ref{lem-Limits}. Taking the limit as $r \to +\infty$ we deduce that $b_{N-k-1}=0$, thanks to Lemmas \ref{lem-Limits} and \ref{lem-LimitDelta^iU}. This completes the present proof.
\end{proof}


Using Lemma \ref{lem-Delta^kU}, we obtain the following representation of $\Delta u$ as follows.
\begin{equation}\label{eqDeltaU}
\Delta  u(x) = c_1 \int_{\Rset^{2N-1}} \frac{ u^{-(4N-1)}(y) }{ |x-y| } dy
\end{equation}
with the constant $c_1$ given in \eqref{eqConstantsC}. Then using \eqref{eqDeltaU}, we obtain a representation for $u$ as the following.
 
\begin{lemma}\label{lem-RepresentationU}
There exists a constant $\gamma$ such that $u$ has the following representation
\begin{equation}\label{eqU}
u(x) = c_0 \int_{\Rset^{2N-1}} {|x - y| u^{-(4N-1)}(y)dy} + \gamma
\end{equation}
with the constant $c_0$ given by \eqref{eqConstantsC0}.
\end{lemma}

\begin{proof}
Denote by $h$ the following function
\[
h(x) = c_0 \int_{\Rset^{2N-1}} {|x - y| u^{-(4N-1)}(y)dy}
 \]
and let
\[
\beta = c_0 \int_{\Rset^{2N-1}} {u^{ - (4N-1)}(y)dy} .
\]
First of all, we have
\[|
\nabla h|(x) = \Big| c_0 \int_{\Rset^{2N-1}} {\frac{x - y}{|x - y|}u^{ - (4N-1)}(y)dy} \Big| \leqslant \beta .
\]
By observing \eqref{eqConstantsC0}, we easily verify that $c_0 \Delta _x (|x - y|) = c_1 |x - y|^{-1}$. From this, it is immediate to see that $\Delta (u - h) = 0$. It follows from the dominated convergence theorem that
\[\lim_{|x| \to +\infty } \frac{h(x)}{|x|} = \beta .\]
Since both $u$ and $h$ are at most linear growth at infinity, we obtain by the generalized Liouville theorem that
\begin{equation}\label{eqRepresentationUTemp} 
u(x) = h(x) +\sum\limits_{i = 1}^{2N-1} b_i x_i + \gamma
\end{equation}
for some constants $b_i$ and $\gamma$. Denote $x/|x|$ and $(b_1, ..., b_{2N-1})$ by $\Theta$ and $\vec b$, respectively. It follows from \eqref{eqRepresentationUTemp} that
\begin{equation}\label{eqQuotientRepresentationUTemp} 
\frac{{u(x)}}{{|x|}} = \frac{{h(x)}}{{|x|}} + \vec b \cdot \Theta + \frac{\gamma }{{|x|}} .
\end{equation}
Taking the limit as $|x| \to +\infty$ to the both sides of \eqref{eqQuotientRepresentationUTemp} we get $\alpha=\beta$ and $\vec b= 0$. This finishes the proof of the lemma.
\end{proof}

In the last part of the section, we prove that $\gamma = 0$. 

\begin{lemma}\label{lem-GammaIsZero}
The constant $\gamma$ in the representation formula \eqref{eqU} is zero.
\end{lemma}

\begin{proof}
An immediate consequence of Lemma \ref{lem-RepresentationU}, we obtain the representation for $\nabla u$ as follows.
\begin{equation}\label{eqNablaU}
\nabla u(x) =c_0 \int_{\Rset^{2N-1}} {\frac{x-y}{|x - y|} u^{-(4N-1)}(y)dy}
\end{equation}
From this we obtain
\begin{equation}\label{eqXdotNablaU}
x \cdot \nabla u(x) = c_0 \int_{ \Rset^{2N-1} } {\frac{{|x{|^2} - x \cdot y}}{{|x - y|}}u^{ - (4N-1)}(y)dy} .
\end{equation}
Now multiply Eq. \eqref{eqXdotNablaU} thoughout by $u^{-(4N-1)}$ and integrate the resulting equation over the ball centered at the origin with radius $R$ to obtain
\[\begin{split}
- \frac{1}{4N-2} & \int_{B(0, R)} {x \cdot \nabla {u^{ - (4N-2)}}(x)dx} \\
& =c_0 \int_{\Rset^{2N-1}}  \Big( {\int_{ B(0, R) } {\frac{{|x{|^2} - x \cdot y}}{{|x - y|}}u^{ - (4N-1)}(x)dx}   \Big)u^{ - (4N-1)}(y)dy} .
\end{split}\]
Now for the left hand side of the preceding equation, we integrate by parts to get
\begin{equation}\label{eqIntegralXdotNablaU}
\begin{split}
  - \frac{1}{4N-2} & \int_{B(0, R)} {x \cdot \nabla {u^{ - (4N-2)}}(x)dx} \\
  = & - \frac{1}{4N-2}\left[ 
\begin{split}
&R\int_{\partial B(0, R)} {{u^{ -(4N-2)}}(x)d{\sigma _x}}\\
& - (2N-1)\int_{B(0, R)} {{u^{ -(4N-2)}}(x)dx} 
\end{split}
\right] \hfill \\
  =& \frac{1}{2}\int_{B(0, R)} {{u^{ - (4N-2)}}(x)dx} - \frac{R}{4N-2}\int_{\partial B(0, R)} {{u^{ -(4N-2)}}(x)d{\sigma _x}} . 
\end{split} 
\end{equation}
For the right hand side, we notice that $|x|^2 - x \cdot y =  \big( {|x - y{|^2} + (x - y) \cdot (x + y)} \big)/2$ which leads to
\begin{align*}
c_0 \int_{ \Rset^{2N-1} } &\Big( {\int_{B(0, R)} {\frac{{|x{|^2} - x \cdot y}}{{|x - y|}}u^{ - (4N-1)}(x)dx}  \Big)u^{ - (4N-1)}(y)dy} \hfill \\
  = &\frac{c_0}{2} \int_{ \Rset^{2N-1} }  \Big( {\int_{B(0, R)} {\frac{{|x - y{|^2} + |x{|^2} - |y{|^2}}}{{|x - y|}}u^{ - (4N-1)}(x)dx}   \Big)u^{ - (4N-1)}(y)dy} \hfill \\
  =& \frac{1}{2}\int_{B(0, R)} {(u(x) - \gamma )u^{ - (4N-1)}(x)dx} \\
&+ \frac{c_0}{2}\int_{ \Rset^{2N-1} }  \Big( {\int_{B(0, R)} {\frac{{|x{|^2} - |y{|^2}}}{{|x - y|}}u^{ - (4N-1)}(x)dx}   \Big)u^{ - (4N-1)}(y)dy} . 
\end{align*} 
Here in the last step, we have used the representation formula for $u$ established in Lemma \ref{lem-RepresentationU}. Letting $R \to +\infty$, since the integrand in the last term is absolutely integrable, this term becomes $\int_{\Rset^{2N-1}}\int_{\Rset^{2N-1}}$ with the same integrand. Hence, in the limit, this last term vanishes. Since $u$ has exact linear growth at infinity and $N \geqslant 2$, the boundary term in Eq. \eqref{eqIntegralXdotNablaU} also vanishes. Hence, one gets
\[
\frac{1}{2}\int_{ \Rset^{2N-1} } {u^{ - (4N-2)}(x)dx} = \frac{1}{2}\int_{ \Rset^{2N-1} } {u^{ -(4N-2)}(x)dx} - \frac{\gamma }{2}\int_{ \Rset^{2N-1} } {u^{ - (4N-1)}(x)dx}, 
\]
which implies $\gamma=0$.
\end{proof}

\begin{proof}[Proof of Theorem \ref{thmCLASSIFICATION}]

Now we prove Theorem \ref{thmCLASSIFICATION}. Suppose that $u$ solves Eq. \eqref{eqMAIN}. Then the representation
\[
u(x) = c_0 \int_{\Rset^{2N-1}} {|x - y|{u^{-(4N-1)}}(y)dy} .
\]
for some positive constant $c_0$ is simply a consequence of Lemmas \ref{lem-RepresentationU} and \ref{lem-GammaIsZero}. From this representation, we can apply a general classification result due to Li in \cite{li2004} to conclude that $u$ takes the following form
\[
u(x) = (1+|x|^2)^{1/2}
\]
in $\Rset^{2N-1}$ up to dilations and translations.
\end{proof}

\section{Non-existence results: Proof of Theorems \ref{thmNONEXISTENCE} and \ref{thmNewNONEXISTENCE}}

\subsection{Proof of Theorem \ref{thmNONEXISTENCE}}

We prove the non-existence result in Theorem \ref{thmNONEXISTENCE} by way of contradiction. Indeed, suppose that $u$ solves Eq. \eqref{eqMAIN2} with exact linear growth $\alpha>0$ at infinity. By the equation, we note that \[
(-\Delta)^N u >0
\]
everywhere in $\Rset^{2N-1}$. Therefore, as in Lemma \ref{lem-PolySubHarmonic}, we can apply a general result from \cite[Theorem 2]{ngo2016} to get
\[
(-\Delta)^k u < 0
\]
everywhere in $\Rset^{2N-1}$ for each $k=1,...,N-1$. In particular $\Delta u<0$ which implies that
\[
\overline u'(r)<0
\]
for any $r$. Since $u$ has exact linear growth $\alpha>0$ at infinity, we deduce that
\[
u(x) \geqslant \frac \alpha 2 |x|
\]
for $|x|$ large. Hence
\[
\overline u(r) = \fint_{\partial B(0, r)} u(x) d\sigma_x \geqslant \frac \alpha 2 r
\]
for large $r$. This gives us a contradiction since $\overline u'<0$.

\subsection{Proof of Theorem \ref{thmNewNONEXISTENCE}}

We prove Theorem \ref{thmNewNONEXISTENCE} by contradiction. First, by contradiction assumption, we recover the super poly-harmonic property for solutions of \eqref{eqMAIN2} without using the linear growth property as in Lemma \ref{lem-PolySubHarmonic}. Indeed, suppose that $u$ solves \eqref{eqMAIN2} which satisfies all assumptions in the theorem, that is
\begin{equation}\label{eqLowerBound}
u(x) \geqslant 1=u(0)
\end{equation}
for all $x \in \Rset^{2N-1}$, 
\begin{equation}\label{eqFiniteEnergy}
\int_{\Rset^{2N-1}} u^{-(4N-2)} dx < +\infty,
\end{equation}
and
\begin{equation}\label{eqFiniteFinite}
\int_{\Rset^{2N-1}} |(-\Delta)^i u|^2 dx < +\infty
\end{equation}
for $i=1,..., N-1$. In the sequel, we prove that there exists a sequence of non-negative functions $U_k$ and a sequence of positive numbers $q_k >1$ such that
\[
(-\Delta)^k u = U_k 
\]
for all $k = 1, ..., N-1$ and that
\[
U_k \in L^q (\Rset^{2N-1})
\]
for all $q > q_k$. By induction, we first verify the statement for $k=N-1$. Set
\[
U_{N-1} (x) = c_{N-1} \int_{\Rset^{2N-1}} \frac{u^{-(4N-1)}(y)}{|x-y|^{2N-3}} dy,
\]
where $c_{N-1}$ is given in \eqref{eqConstantsC}. Thanks to \eqref{eqLowerBound} and \eqref{eqFiniteEnergy}, it is not hard to see that $$\int_{\Rset^{2N-1}} u^{-q} (x) dx < +\infty $$ for all $q \geqslant 4N-2$; hence $U_{N-1} \in L^{q}(\Rset^{2N-1})$ for all $q >1 = : q_{N-1}$. As in the proof of Lemma \ref{lem-Delta^(N-1)}, there holds
\begin{equation}\label{Laplace}
\Delta ((-\Delta)^{N-1}u - U_{N-1} ) = 0.
\end{equation}
On the other hand, for $r>0$ and any  $x \in \Rset^{2N-1}$, we have
\begin{align}\label{FirstChange}
\int_{B(x, r)} u^{-(4N-1)} dy 
&= - r^{2N-2} \frac{\partial}{\partial r} \Big( r^{-(2N-2)} \int_{\partial B(x, r)} (-\Delta)^{N-1} u d\sigma \Big)  .
\end{align}
After dividing both sides of \eqref{FirstChange} by $r^{2N-2}$ and integrating the resulting equation over $[0,r]$, we obtain
\begin{equation}\label{SecondChange}
\begin{split}
\int_{0}^{r} s_1^{-(2N-2)} & \Big(\int_{B(x,s_1)} u^{-(4N-1)} dy \Big)ds_1  \\
&= -  r^{-(2N-2)} \int_{\partial B(x,r)} (-\Delta)^{N-1} u d\sigma + \omega_{2N-1} (-\Delta)^{N-1} u(x).
\end{split}
\end{equation}
Multiplying both sides of \eqref {SecondChange} by $r^{2N-2}$ and integrating the result equation over $[0,r]$ to get
\begin{equation}\label{ThirdChange}
\begin{split}
\int_{0}^{r} s_2^{2N-2}  & \Big( \int_0^{s_2} s_1^{-(2N-2)} \Big(\int_{B(x, s_1)} u^{-(4N-1)} dy \Big)ds_1\Big) ds_2  \\
= &- \int_{B(x, r)} (-\Delta)^{N-1} u dy + \frac{\omega_{2N-1} }{2N-1}(-\Delta)^{N-1} u(x) r^{2N-1}\\
= &r^{2N-2} \frac{\partial}{\partial r}\Big( r^{-(2N-2)}  \int_{\partial B(x, r)} (-\Delta)^{N-2} u d\sigma \Big) \\
&+ \frac{\omega_{2N-1} }{2N-1}(-\Delta)^{N-1} u(x) r^{2N-1}. 
\end{split}
\end{equation}
Repeating the above argument to get
\begin{equation}\label{FourthChange}
\begin{split}
g(r) :=& \int_0^r s_3^{-(2N-2)} \Big( \int_{0}^{s_3} s_2^{2N-2}  \Big( \int_0^{s_2} s_1^{-(2N-2)} \Big(\int_{B(x, s_1)} u^{-(4N-1)} dy \Big)ds_1\Big) ds_2 \Big) ds_3 \\
= &r^{-(2N-2)}  \int_{\partial B(x, r)} (-\Delta)^{N-2} u d\sigma  + \omega_{2N-1} (-\Delta)^{N-2} u(x) \\
 & +\frac{\omega_{2N-1}}{2 (2N-1)} (-\Delta)^{N-1} u(x) r^2.
\end{split}
\end{equation}
Making use of the L'Hospital rule, we conclude that  
\begin{equation*}
\lim_{r \to +\infty} \frac{g(r)}{r^2} \leqslant C,
\end{equation*}
for some constant $C>0$ independent of $x$. Back to \eqref{FourthChange} to conclude that $(-\Delta)^{N-1} u$ is bounded from above in $ \Rset^{2N-1}$. Together with the fact that $U_{N-1}$ is positive everywhere, by the Liouville theorem, we obtain from \eqref{Laplace} that 
\begin{equation*}
(-\Delta)^{N-1} u - U_{N-1} = C,
\end{equation*}
everywhere in $\Rset^{2N-1}$ for some constant $C$. Since $U_{N-1} \in L^q(\Rset^{2N-1})$ for any $q>q_{N-1}$, we claim that $\lim_{|x| \to \infty} U_{N-1} (x) =0$. This combines with the condition \eqref{eqFiniteFinite} gives $C=0$. That is equivalent to 
\[
(-\Delta )^{N-1} u = U_{N-1} \geqslant 0. 
\]

Now, we suppose that 
\begin{equation*}
(-\Delta)^{N-k} u = U_{N-k}
\end{equation*}
for some non-negative function $U_{N-k}$ in $\Rset^{2N-1}$ with $U_{N-k} \in L^q(\Rset^{2N-1})$ for any $q>q_k$ for some positive constant $q_k$. Our next task is to prove that $(-\Delta)^{N-k-1} u$ has the similar property. To this purpose, we repeat the same calculation as above. Indeed, we set
\begin{equation*}
U_{N-k-1}(x) = c_{N-1} \int_{\Rset^{2N-1}} \frac {U_{N-k}(y)}{|x-y|^{2N-3}}dy.
\end{equation*}
Hence, similar to the way to obtain \eqref{FourthChange}, after several steps we arrive at
\[
\begin{split}
\int_0^r s_3^{-(2N-2)}  & \Big( \int_{0}^{s_3} s_2^{2N-2}  \Big( \int_0^{s_2} s_1^{-(2N-2)} \Big(\int_{B(x, s_1)} U_{N-k}(y)dy \Big) ds_1\Big) ds_2 \Big) ds_3 
 \\
= & r^{-(2N-2)}  \int_{\partial B(x, r)} (-\Delta)^{N-k-2}u d\sigma  + \omega_{2N-1} (-\Delta)^{N-k-2}u(x) \\
& + \frac{\omega_{2N-1} }{2 (2N-1)}(-\Delta)^{N-k-1} u(x) r^2.
\end{split}
\]
From this, it is not hard to see that the function $(-\Delta)^{N-k-1}u$ is bounded from above and 
\begin{align*}
\Delta U_{N-k-1}(x) 
&= - U_{N-k}(x) = - (-\Delta)^{N-k} u(x).  
\end{align*}
Therefore,
\begin{equation*}
\Delta ( (-\Delta)^{N-k-1}u- U_{N-k-1}) = 0.
\end{equation*}
From the positivity of $U_{N-k-1}$, we get that $ (-\Delta )^{N-k-1}u - U_{N-k-1} $ is also bounded from above. Therefore, by the Liouville theorem, there exists a constant $C$ such that
\begin{equation*}
 (-\Delta)^{N-k-1} u - U_{N-k-1} = C,
\end{equation*}
everywhere in $\Rset^{2N-1}$. Meanwhile, since $(-\Delta)^{N-k}u = U_{N-k} \in L^q(\Rset^{2N-1})$ for any $q>q_k$, we can conclude that there exists some $q_{k+1} > q_k$ such that $U_{N-k-1} \in L^q(\Rset^{2N-1})$ for any $q>q_{k+1}$. Hence, there holds $C=0$, which completes the proof of the statement. 

Let $k=N-1$, it follows that $-\Delta u$ is non-negative. However, we can also check that
\[
\Delta\Big( \frac{1}{u} \Big) = - \frac{\Delta u}{u^2} + 2 \frac{|\nabla u|^2}{u^3} \geqslant 0. 
\]
It follows that $1/u$ must be constant, which contradicts with \eqref{eqFiniteEnergy}. The proof is complete.

\section*{Acknowledgments}

The author is deeply grateful to Nguyen Tien Tai for useful discussion on the non-existence part of the paper.

\end{document}